\documentclass[11pt]{article}

\usepackage{amssymb, amsmath, amsthm, graphicx}
\usepackage[left=1in,top=1in,right=1in]{geometry}

\date{}

\theoremstyle{plain}
      \newtheorem{theorem}{Theorem}[section]
      \newtheorem{lemma}[theorem]{Lemma}

      \newtheorem{conjecture}[theorem]{Conjecture}
\theoremstyle{definition}

\theoremstyle{remark}

	\newcommand{\RR}{{\mathbb R}}

\title{Erd\H os-Hajnal conjecture for graphs with bounded VC-dimension}
\author{Jacob Fox \thanks{Stanford University, Stanford, CA. Supported by a Packard Fellowship, by NSF CAREER award DMS-1352121, and by an Alfred P. Sloan Fellowship. Email: {\tt jacobfox@stanford.edu}.} \and J\'anos Pach\thanks{EPFL, Lausanne and R\'enyi Institute, Budapest.  Research partially supported by Swiss National Science Foundation grants 200020-144531, 200021-137574, and 200021-162884. Email:
{\tt pach@cims.nyu.edu}.}\and Andrew Suk\thanks{Department of Mathematics,  University of California at San Diego, La Jolla, CA, 92093 USA. Supported by NSF grant DMS-1500153, an NSF CAREER award, and an Alfred Sloan Fellowship. Email: {\tt asuk@ucsd.edu}.}}

\begin{document}

\maketitle

\begin{abstract}
The {\em Vapnik-Chervonenkis dimension} (in short, VC-dimension) of a {\em graph} is defined as the VC-dimension of the set system induced by the neighborhoods of its vertices.  We show that every $n$-vertex graph with bounded VC-dimension contains a clique or an independent set of size at least $e^{(\log n)^{1 - o(1)}}$. The dependence on the VC-dimension is hidden in the $o(1)$ term. This improves the general lower bound, $e^{c\sqrt{\log n}}$, due to Erd\H os and Hajnal, which is valid in the class of graphs satisfying any fixed nontrivial hereditary property. Our result is almost optimal and nearly matches the celebrated Erd\H os-Hajnal conjecture, according to which one can always find a clique or an independent set of size at least $e^{\Omega(\log n)}$. Our results partially explain why most geometric intersection graphs arising in discrete and computational geometry have exceptionally favorable Ramsey-type properties.

Our main tool is a partitioning result found by Lov\'asz-Szegedy and Alon-Fischer-Newman, which is called the ``ultra-strong regularity lemma'' for graphs with bounded VC-dimension. We extend this lemma to $k$-uniform hypergraphs, and prove that the number of parts in the partition can be taken to be $(1/\varepsilon)^{O(d)}$, improving the original bound of $(1/\varepsilon)^{O(d^2)}$ in the graph setting. We show that this bound is tight up to an absolute constant factor in the exponent. Moreover, we give an $O(n^k)$-time algorithm for finding a partition meeting the requirements.  Finally, we establish tight bounds on Ramsey-Tur\'an numbers for graphs with bounded VC-dimension.
\end{abstract}

\section{Introduction}

During the relatively short history of computational geometry, there were many breakthroughs that originated from results in extremal combinatorics \cite{GRT17}. Range searching turned out to be closely related to discrepancy theory \cite{Ch00}, linear programming to McMullen's Upper Bound theorem and to properties of the facial structure of simplicial complexes \cite{St96}, motion planning to the theory of Davenport-Schinzel sequences and to a wide variety of other forbidden configuration results \cite{ShA95}, graph drawing and VLSI design to the crossing lemma, to the Szemer\'edi-Trotter theorem, and to flag algebras \cite{Ta13}. A particularly significant example that found many applications in discrete and computational geometry, was the discovery of Haussler and Welzl \cite{HW87}, according to which many geometrically defined set systems have bounded Vapnik-Chervonenkis dimension. Erd\H os's ``Probabilistic Method'' \cite{AS} or ``Random Sampling'' techniques, as they are often referred to in computational context, had been observed to be ``unreasonably effective'' in discrete geometry and geometric approximation algorithms \cite{H11}. Haussler and Welzl offered an explanation and a tool: set systems of bounded Vapnik-Chervonenkis dimension admit much smaller hitting sets and ``epsilon-nets'' than other set systems with similar parameters.

It was also observed a long time ago that geometrically defined graphs and set systems have unusually strong Ramsey-type properties. According to the quantitative version of Ramsey's theorem, due to Erd\H os and Szekeres \cite{es}, every graph on $n$ vertices contains a clique or an independent set of size at least $\frac{1}{2}\log n$.  In \cite{erdos2}, Erd\H os proved that this bound is tight up to a constant factor. However, every intersection graph of $n$ segments in the plane, say, has a much larger clique or an independent set, whose size is at least $n^{\varepsilon}$ for some $\varepsilon>0$ \cite{LMPT}. The proof extends to intersection graphs of many other geometric objects \cite{alon}. Interestingly, most classes of graphs and hypergraphs in which a similar phenomenon has been observed turned out to have (again!) bounded Vapnik-Chervonenkis dimension. (We will discuss this fact in a little more detail at the end of the Introduction.)

The problem can be viewed as a special case of a celebrated conjecture of Erd\H os and Hajnal \cite{hajnal}, which is one of the most challenging open problems in Ramsey theory. Let $P$ be a {\em hereditary} property of finite graphs, that is, if $G$ has property $P$, then so do all of its induced subgraphs. Erd\H os and Hajnal conjectured that for every hereditary property $P$ which is not satisfied by all graphs, there exists a constant $\varepsilon(P)>0$ such that every graph of $n$ vertices with property $P$ has a clique or an independent set of size at least $n^{\varepsilon(P)}$. They proved the weaker lower bound $e^{\varepsilon(P)\sqrt{\log n}}$. According to the discovery of Haussler and Welzl mentioned above, the Vapnik-Chervonenkis dimension of most classes of ``naturally'' defined graphs arising in geometry is bounded from above by a constant $d$. The property that the Vapnik-Chervonenkis dimension of a graph is at most $d$, is hereditary.

The aim of this paper is to investigate whether the observation that the Erd\H os-Hajnal conjecture tends to hold for geometrically defined graphs can be ascribed to the fact that they have bounded VC-dimension.  Our first theorem (Theorem 1 below) shows that the answer to this question is likely to be positive. To continue, we need to agree on the basic definitions and terminology.

Let $\mathcal{F}$ be a set system on a ground set $V$.  The {\em Vapnik-Chervonenkis dimension} ({\em VC-dimension}, for short) of $\mathcal{F}$ is the {\em largest} integer $d$ for which there exists a $d$-element set $S\subset V$ such that for every subset $B\subset S$, one can find a member $A\in \mathcal{F}$ with $A\cap S=B$.  Given a graph $G = (V,E)$, for any vertex $v \in V$, let $N(v)$ denote the neighborhood of $v$ in $G$, that is, the set of vertices in $V$ that are connected to $v$.  We note that $v$ itself is not in $N(v)$.  Then we say that $G$ has \emph{VC-dimension}~$d$, if the set system induced by the neighborhoods in $G$, i.e. $\mathcal{F} = \{N(v) \subset V: v \in V\}$, has VC-dimension $d$.  Let us remark that although the edges of $G$ also form a 2-uniform set system $\mathcal{F}' = \{e\in E(G)\}$, the VC-dimension of $G$ defined above is usually different from the VC-dimension of $\mathcal{F}'$.

The VC-dimension of a set system is one of the most useful combinatorial parameters that measures its complexity, and, apart from its geometric applications, it has proved to be relevant in many other branches of pure and applied mathematics, such as statistics, logic, learning theory, and real algebraic geometry. The notion was introduced by Vapnik and Chervonenkis~\cite{VC71} in 1971, as a tool in mathematical statistics. Kranakis et al.~\cite{KKRUW} observed that the VC-dimension of a graph can be determined in quasi-polynomial time and, for bounded degree graphs, in quadratic time. Schaefer \cite{Sch99}, addressing a question of Linial, proved that determining the VC-dimension of a set system is $\Sigma_3^p$-complete. For each positive integer $d$, Anthony, Brightwell, and Cooper \cite{ABC95} determined the threshold for the Erd\H{o}s-R\'enyi random graph $G(n,p)$ to have VC-dimension $d$ (see also~\cite{KPW92}). Given a bipartite graph $F$, its \emph{closure} is defined as the set of all graphs that can be obtained from $F$ by adding edges between two vertices in the same part.  It is known (see \cite{LS}) that a class of graphs has bounded VC-dimension if and only if none of its members contains any induced subgraph that belongs to the closure of some fixed bipartite graph $F$.

Our first result states that the Erd\H os-Hajnal conjecture ``almost holds'' for graphs of bounded VC-dimension.

\begin{theorem}\label{jacob}
Let $d$ be a fixed positive integer. If $G$ is an $n$-vertex graph with VC-dimension at most $d$, then $G$ contains a clique or independent set of size $e^{(\log n)^{1 - o(1)}}$.
\end{theorem}

\noindent Note that the dependence of the bound on $d$ is hidden in the $o(1)$-notation.

There has been a long history of studying off-diagonal Ramsey numbers, where one is interested in finding the maximum size of an independent set guaranteed in a $K_s$-free graph on $n$ vertices with $s$ fixed.  An old result of Ajtai, Koml\'os, and Szemer\'edi \cite{AKS} states that all such graphs contain independent sets of size $cn^{\frac{1}{s-1}}(\log n)^{\frac{s-2}{s-1}}$.  In the other direction, Spencer \cite{Sp} used the Lov\'asz Local Lemma to show that there are $K_s$-free graphs on $n$ vertices and with no independent set of size $c'n^{\frac{2}{s+1}}\log n$.  This bound was later improved by Bohman and Keevash~\cite{BK} to $c'n^{\frac{2}{s+1}}(\log n)^{1 - \frac{2}{(s + 1)(s - 2)}}$.  In Section \ref{lll}, we give a simple proof, extending Spencer's argument, showing that there are $K_s$-free graphs with bounded VC-dimension and with no large independent sets.

\begin{theorem}\label{offdiag}

For fixed $s\geq 3$ and $d \geq 5$ such that $d \geq s+2$, there exists a $K_s$-free graph on $n$ vertices with VC-dimension at most $d$ and no independent set of size $cn^{\frac{2}{s+1}}\log n$, where $c = c(d)$.

\end{theorem}

\noindent For large $s$ ($s > d$), a result of Fox and Sudakov (Theorem 1.9 in \cite{FS}) implies that all $n$-vertex $K_s$-free graphs $G$ with VC-dimension $d$ contain an independent set of size $n^{\frac{1}{c\log s}}$ where $c = c(d)$.

\bigskip

\noindent \textbf{Regularity lemma for hypergraphs with bounded VC-dimension.}  First, we generalize the definition of VC-dimension for graphs to {\em hypergraphs}.  Given a $k$-uniform hypergraph $H = (V,E)$, for any $(k-1)$-tuple of distinct vertices $v_1,\ldots, v_{k-1} \in V$, let $$N(v_1,\ldots, v_{k-1}) = \{u \in V: \{v_1,\ldots, v_{k-1},u\} \in E(H)\}.$$  Then we say that {\em $H$ has VC-dimension $d$}, if the set system $$\mathcal{F} = \{N(v_1,\ldots, v_{k-1})\subset V: v_1,\ldots, v_{k-1} \in V\}$$ has VC-dimension $d$.  Of course, the hyperedges of $H$ form a set system, but the VC-dimension of this set system is usually {\em different} from the VC-dimension of $H$ defined above. The latter one is defined as the VC-dimension of the set system $\mathcal{F}$ induced by the neighborhoods of the vertices of $H$, rather than by the hyperedges.

The \emph{dual} of the set system $(V,\mathcal{F})$ on the ground set $V$ is the set system obtained by interchanging the roles of $V$ and $\mathcal{F}$. That is, it is the set system $(\mathcal{F},\mathcal{F}^{\ast})$, where the ground set is $\mathcal{F}$  and
$$\mathcal{F}^{\ast} = \{\{A\in \mathcal{F}: v\in A\}:v \in V\}.$$
In other words, $\mathcal{F}^{\ast}$ is isomorphic to the set system whose ground set is ${V\choose k-1}$, and each set is a maximal collection of $(k-1)$-tuples $\{S_1,\ldots, S_p\}$ such that for all $i$, $v\cup S_i \in E(H)$ for some fixed $v$. Hence, we have $(\mathcal{F}^{\ast})^{\ast} = \mathcal{F}$, and it is known that if $\mathcal{F}$ has VC-dimension $d$, then $\mathcal{F}^{\ast}$ has VC-dimension at most $2^d + 1$.  We say that $H =(V,E)$ has \emph{dual VC-dimension} $d$ if $\mathcal{F}^{\ast}$ has VC-dimension $d$.

The main tool used to prove Theorem \ref{jacob} is an ultra-strong regularity lemma for graphs with bounded VC-dimension obtained by Lov\'asz and Szegedy \cite{LS} and Alon, Fischer, and Newman \cite{AFN}.  Here, we extend the ultra-strong regularity lemma to uniform hypergraphs.

Given $k$ vertex subsets $V_1,\ldots, V_k$ of a $k$-uniform hypergraph $H$, we write $E(V_1,\ldots, V_k)$ to be the set of edges going across $V_1,\ldots, V_k$, that is, the set of edges with exactly one vertex in each $V_i$.  The \emph{density} across $V_1,\ldots, V_k$ is defined as $\frac{|E(V_1,\ldots, V_k)|}{|V_1|\cdots |V_k|}$.  We say that the $k$-tuple $(V_1,\ldots, V_k)$ is \emph{$\varepsilon$-homogeneous} if the density across it is less than $\varepsilon$ or greater than $1-\varepsilon$.  A partition is called {\em equitable} if any two parts differ in size by at most one.

In \cite{LS}, Lov\'asz and Szegedy established an \emph{ultra-strong} regularity lemma for graphs ($k = 2$) with bounded VC-dimension, which states that for any $\varepsilon > 0$, there is a (least) $K  = K(\varepsilon)$ such that the vertex set $V$ of a graph with VC-dimension $d$ has a partition into at most $K\leq (1/\varepsilon)^{O(d^2)}$ parts such that all but at most an $\varepsilon$-fraction of the pairs of parts are $\varepsilon$-homogeneous.  A better bound was obtained by Alon, Fischer, and Newman \cite{AFN} for bipartite graphs with bounded VC-dimension, who showed that the number of parts in the partition can be taken to be $(d/\varepsilon)^{O(d)}$.  Since the VC-dimension of a graph $G$ is equivalent to the dual VC-dimension of $G$, we generalize their result to hypergraphs with the following result.

\begin{theorem}\label{reg1}

Let $\varepsilon \in (0,1/4)$ and let $H = (V,E)$ be an $n$-vertex $k$-uniform hypergraph with dual VC-dimension $d$.  Then $V$ has an equitable partition $V = V_1\cup\cdots\cup V_K$ with $8/\varepsilon \leq K \leq c(1/\varepsilon)^{2d + 1}$ parts such that all but an $\varepsilon$-fraction of the $k$-tuples of parts are $\varepsilon$-homogeneous. Here $c = c(d,k)$ is a constant depending only on $d$ and $k$.  Moreover, there is an $O(n^k)$ time algorithm for computing such a partition.

\end{theorem}

Our next result shows that the partition size in the theorem above is tight up to an absolute constant factor in the exponent.

\begin{theorem}\label{tight}
For $d\geq 16$ and $\varepsilon \in (0,1/100)$, there is a graph $G$ with VC-dimension $d$ such that any equitable vertex partition on $G$ with the property that all but an $\varepsilon$-fraction of the pairs of parts are $\varepsilon$-homogeneous, requires at least $(5\varepsilon)^{-d/4}$ parts.

\end{theorem}

\bigskip

\noindent \textbf{Ramsey-Tur\'an numbers.}  Let $F$ be a fixed graph.  The Ramsey-Tur\'an number $\mathbf{RT}(n,F,o(n))$ is the maximum number of edges an $n$-vertex graph $G$ can have without containing $F$ as a subgraph and having independence number $o(n)$.  Ramsey-Tur\'an numbers were introduced by S\'os \cite{sos}, motivated by the classical theorems of Ramsey and Tur\'an and their connections to geometry, analysis, and number theory.  One of the earliest results in Ramsey-Tur\'an theory appeared in \cite{ES70}. It states that for $p\geq 2$, we have

$$\mathbf{RT}(n,K_{2p-1},o(n)) = \frac{1}{2}\left(1 - \frac{1}{p-1}\right)n^2 + o(n^2).$$

\noindent For the case when the excluded clique has an even number of vertices,  Szemer\'edi \cite{Sz72} applied the graph regularity lemma to show that

$$\mathbf{RT}(n,K_4,o(n)) \leq \frac{1}{8}n^2 + o(n^2),$$

\noindent and several years later, Bollob\'as--Erd\H{o}s \cite{BE76} gave a surprising geometric construction which shows that this bound is tight.  For larger cliques, a result of Erd\H os, Hajnal, S\'os, and Szemer\'edi \cite{EHSS} states that

$$\mathbf{RT}(n,K_{2p},o(n)) = \frac{1}{2}\left(1-\frac{3}{3p-2}\right)n^2 + o(n^2)$$

\noindent holds for every $p\ge 2$.  For more results in Ramsey-Tur\'an theory, see the survey of Simonovits and S\'os \cite{SS}.

Here we give tight bounds on Ramsey-Tur\'an numbers for graphs with bounded VC-dimension, showing that the densities for $K_{2p}$ and for $K_{2p-1}$ are the same in this setting, and are different from what we have in the classical setting in the even case.

Let $\mathbf{RT}_{d}(n,K_p,o(n))$ be the maximum number of edges that an $n$-vertex $K_p$-free graph of VC-dimension at most $d$ can have if its independence number is $o(n)$.

\begin{theorem}\label{rt}

For fixed integers $d\geq 4$ and $p \geq 3$, we have

$$\mathbf{RT}_{d}(n,K_{2p-1},o(n)) = \mathbf{RT}_{d}(n,K_{2p},o(n)) = \frac{1}{2}\left(1 - \frac{1}{p-1}\right)n^2  + o(n^2).$$

\end{theorem}

\bigskip

\noindent \textbf{Semi-algebraic graphs vs. graphs with bounded VC-dimension.}  A \emph{semi-algebraic} graph $G$, is a graph whose vertices are points in $\RR^d$ and edges are pairs of points that satisfy a semi-algebraic relation of constant complexity.\footnote{A binary semi-algebraic relation $E$ on a point set $P\subset \RR^d$ is the set of pairs of points $(p,q)$ from $P$ whose $2d$ coordinates satisfy a boolean combination of a fixed number of polynomial inequalities.}   In a sequence of recent works \cite{alon,CFPSS,FPS14}, several authors have shown that classical Ramsey and Tur\'an-type results in combinatorics can be significantly improved for semi-algebraic graphs.

It follows from the Milnor-Thom theorem (see \cite{mat}) that semi-algebraic graphs of bounded complexity have bounded VC-dimension.  Therefore, all results in this paper on properties of graphs of bounded VC-dimension apply to semi-algebraic graphs of bounded description complexity. However, a graph being semi-algebraic of bounded complexity is a much more restrictive condition than having bounded VC-dimension. In particular, it is known (it follows, e.g., from \cite{ABC95}) that for each $\varepsilon>0$ there is a positive integer $d=d(\varepsilon)$ such that the number of $n$-vertex graphs with VC-dimension $d$ is $2^{\Omega(n^{2-\varepsilon})}$, while the Milnor-Thom theorem can be used to deduce that the number of $n$-vertex semi-algebraic graphs coming from a relation with bounded ``description complexity'' is only $2^{O(n\log n)}$. Furthermore, it is known \cite{alon} that semi-algebraic graphs have the {\em strong Erd\H os-Hajnal property}, that is, there exists a constant $\delta>0$ such that every $n$-vertex semi-algebraic graph of bounded complexity contains a complete or an empty {\em bipartite} graph whose parts are of size at least $\delta n$.  This is not true, in general, for graphs with bounded VC-dimension.  In particular, the probabilistic construction in Section \ref{lll} shows the following.

\begin{theorem}\label{notstrongeh}
For fixed $d \geq 5$ and for every sufficiently large $n$, there is an $n$-vertex graph $G = (V,E)$ with VC-dimension at most $d$ with the property that there are no two disjoint subsets $A,B\subset V(G)$ such that $|A|,|B| \geq 4n^{4/d}\log n$ and $(A,B)$ is homogeneous, that is, either $A\times B \subset E(G)$ or $(A\times B) \cap E(G) = \emptyset$.

\end{theorem}

It follows from a result of Alon \emph{et al.}~\cite{alon} that a stronger regularity lemma holds for semi-algebraic graphs of bounded description complexity, where all but an $\varepsilon$-fraction of the pairs of parts in the equitable partition are complete or empty, instead of just $\varepsilon$-homogeneous as in the bounded VC-dimension case (see \cite{PS01}). This result was further extended to $k$-uniform hypergraphs by Fox et al.~\cite{FGLNP}, and the authors \cite{FPS14} recently showed that it holds with a polynomial number of parts.

\medskip

\noindent \textbf{Organization.}  In the next section, we prove Theorem \ref{reg1}.  In Section \ref{graphramsey}, we prove Theorem~\ref{jacob}, which nearly settles the Erd\H os-Hajnal conjecture for graphs with bounded VC-dimension.  In Section~\ref{lll}, we prove Theorems \ref{offdiag} and \ref{notstrongeh}.  In Section \ref{ramseyturan}, we prove Theorem \ref{rt}. We conclude by discussing a number of other results for graphs and hypergraphs with bounded VC-dimension.  We systemically omit floors and ceilings whenever they are not crucial for sake of clarity in our presentation.  All logarithms are natural logarithms.

\section{Regularity partition for hypergraphs with bounded VC-dimension}

In this section, we prove Theorem \ref{reg1}.  We start by recalling several classic results on set systems with bounded VC-dimension.  Let $\mathcal{F}$ be a set system on a ground set $V$.  The \emph{primal shatter function} of $\mathcal{F}$ is defined as

\begin{equation*}
\pi_{\mathcal{F}}(z) = \max\limits_{V'\subset V, |V'| = z}|\{A\cap V':A \in \mathcal{F}\}|.
\end{equation*}

\noindent In other words, $\pi_{\mathcal{F}}(z)$ is a function whose value at $z$ is the maximum possible number of distinct intersections of the sets of $\mathcal{F}$ with a $z$-element subset of $V$.  The \emph{dual shatter function} of $(V,\mathcal{F})$, denoted by $\pi^{\ast}_{\mathcal{F}}$, whose value at $z$ is defined as the maximum number of equivalence classes on $V$ defined by a $z$-element subfamily $\mathcal{Y}\subset \mathcal{F}$, where two points $x,y \in V$ are {\em equivalent} with respect to $\mathcal{Y}$ if $x$ belongs to the same sets of $\mathcal{Y}$ as $y$ does.   In other words, the dual shatter function of $\mathcal{F}$ is the primal shatter function of the dual set system $\mathcal{F}^{\ast}$.

The VC-dimension of $\mathcal{F}$ is closely related to its shatter functions.  A famous result of Sauer \cite{Sa72}, Shelah \cite{Sh72}, Perles, and Vapnik-Chervonenkis \cite{VC71} states the following.

\begin{lemma}\label{sauer}

If $\mathcal{F}$ is a set system with VC-dimension $d$, then

$$\pi_{\mathcal{F}}(z) \leq   \sum_{i=0}^{d}{z\choose i}.$$
\end{lemma}

\noindent On the other hand, suppose that the primal shatter function of $\mathcal{F}$ satisfies $\pi_{\mathcal{F}}(z) \leq cz^d$ for all $z$.  Then, if the VC-dimension of $\mathcal{F}$ is $d_0$, we have $2^{d_0} \leq c(d_0)^d$, which implies $d_0 \leq 4d\log (cd)$.  It is known that if $\mathcal{F}$ has VC-dimension $d$, then $\mathcal{F}^{\ast}$ has VC-dimension at most $2^d + 1$.

Given two sets $A_1,A_2 \in \mathcal{F}$, the {\em symmetric difference} of $A_1$ and $A_2$, denoted by $A_1\triangle A_2$, is the set $(A_1\cup A_2)\setminus(A_1\cap A_2)$.  We say that the set system $\mathcal{F}$ is $\delta$-\emph{separated} if for any two sets $A_1, A_2  \in \mathcal{F}$ we have $|A_1\triangle A_2| \geq \delta$.  The following \emph{packing lemma} was proved by Haussler in \cite{H95}.

\begin{lemma}\label{packing}

Let $\mathcal{F}$ be a set system on a ground set $V$ such that $|V| = n$ and $\pi_{\mathcal{F}}(z) \leq cz^d$ for all $z$.  If $\mathcal{F}$ is $\delta$-separated, then $|\mathcal{F}| \leq c_1(n/\delta)^d$ where $c_1 = c_1(c,d)$.

\end{lemma}

\noindent  We will use Lemma \ref{packing} and the following lemma to prove Theorem \ref{reg1}.

\medskip

\begin{lemma}\label{pattern}

Let $0 < \varepsilon < 1/2$ and $H = (W_1\cup \cdots \cup W_k,E)$ be a $k$-partite $k$-uniform hypergraph such that $|W_i| = m$ for all $i$.  If $(W_1,\ldots, W_k)$ is not $\varepsilon$-homogeneous, then there are at least $\varepsilon(1- \varepsilon) m^{k + 1}$ pairs of $k$-tuples $(e,e')$, where $|e\cap e'| = k-1$, $e \in E(H)$, $e'\not\in E(H)$, and $|e\cap W_i| = |e'\cap W_i| = 1$ for all $i$.

\end{lemma}

  \begin{proof}
 Let $\varepsilon_j$ be the fraction of pairs of $k$-tuples $(e,e')$, each containing one vertex in each $W_i$ and agree on all vertices except in $W_j$, and $e$ is an edge and $e'$ is not an edge. It suffices to show that $\varepsilon_1 + \varepsilon_2 + \cdots + \varepsilon_k\geq \varepsilon(1-\varepsilon)$.

Pick vertices $a_i,b_i \in W_i$ uniformly at random with repetition for $i = 1,2,\ldots, k$.  For $0 \leq i \leq k$, let  $e_i = \{a_j: j \leq i\} \cup \{b_j : j > i\}$.  In particular, $e_k = (a_1,\ldots, a_k)$ and $e_0 = (b_1,\ldots, b_k)$.  Then let $X$ be the event that $e_0$ and $e_k$ have different adjacency, that is, $e_0$ is an edge and $e_k$ is not an edge, or $e_0$ is not an edge and $e_k$ is an edge.  Then we have

$$Pr[X] \geq 2\varepsilon(1 - \varepsilon),$$

\noindent since $(W_1,\ldots, W_k)$ is not $\varepsilon$-homogeneous.  Let $X_i$ be the event that $e_i$ and $e_{i + 1}$ have different adjacency, and let $Y$ be the event that at least one event $X_i$ occurs.  Then by the union bound, we have

$$Pr[Y] \leq Pr[X_0] + Pr[X_1] + \cdots + Pr[X_{k-1}]  = 2\varepsilon_1 + 2\varepsilon_2 + \cdots + 2\varepsilon _k .$$

  On the other hand, if $X$ occurs, then $Y$ occurs.  Therefore $2\varepsilon_1 + 2\varepsilon_2 + \cdots + 2\varepsilon_k \geq Pr[Y] \geq Pr[X] \geq 2(1-\varepsilon)\varepsilon$, which completes the proof.  \end{proof}

\medskip

\noindent \begin{proof}[Proof of Theorem \ref{reg1}]  Let $ 0 < \varepsilon  < 1/4$ , $k \geq 2$, and $H = (V,E)$ be an $n$-vertex $k$-uniform hypergraph with dual VC-dimension $d$. For every vertex $v\in V$, let $N(v)$ denote the set of $(k-1)$-tuples $S \in {V\choose k-1}$ such that $v\cup S \in E(H)$.  Let $\mathcal{F}$ be the set-system whose ground set is ${V \choose k-1}$, and $A \in \mathcal{F}$ if and only if $A = N(v)$ for some vertex $v \in V$.  Hence $\mathcal{F} = \{N(v): v \in V\}$ has VC-dimension $d$.  Set $\delta = \frac{\varepsilon^2}{4k^2}{n\choose k-1}$.  By examining each vertex and its neighborhood one by one, we greedily construct a maximal set $S\subset V(H)$ such that $\mathcal{F}' = \{N(s): s \in S\}$ is $\delta$-separated.  By Lemma \ref{packing}, we have $|S| \leq c_1(4k^2/\varepsilon^2)^d$.  Let $S = \{s_1,s_2,\ldots, s_{|S|}\}$.

We define a partition $\mathcal{Q}: V = U_1\cup \cdots \cup U_{|S|}$ of the vertex set such that $v \in U_i$ if $i$ is the smallest index such that $|N(v)\triangle N(s_i)| < \delta$.  Such an $i$ always exists, since $S$ is maximal.   By the triangle inequality, for $u,v \in U_i$, we have $|N(u)\triangle N(v)| <2\delta$.  Set $K = 8k|S|/\varepsilon$. Partition each part $U_i$ into parts of size $|V|/K = n/K$ and possibly one additional part of size less than $n/K$.   Collect these additional parts and divide them into parts of size $|V|/K$ to obtain an equitable partition $\mathcal{P}:V=V_1 \cup \cdots \cup V_K$ into $K$ parts. The number of vertices of $V$ belonging to parts $V_i$ that are not fully contained in one part of $\mathcal{Q}$ is at most $|S||V|/K$. Hence, the fraction of (unordered) $k$-tuples $(V_{i_1},\ldots, V_{i_k})$ such that at least one of the parts is not fully contained in some part of $\mathcal{Q}$ is at most $k|S|/K=\varepsilon/8$.   Let $X$ denote the set of unordered $k$-tuples of parts $(V_{i_1},\ldots,V_{i_k})$ such that each part is fully contained in a part of $\mathcal{Q}$ (though, in not necessarily the same part) and $(V_{i_1},\ldots,V_{i_k})$ is not $\varepsilon$-homogeneous.

Let $T$ be the set of pairs of $k$-tuples $(e,e')$, such that $|e\cap e'| = k-1$,  $e \in E(H)$, $e'\not\in E(H)$, $|e\cap V_{i_j}| = |e'\cap V_{i_j}| = 1$ for $j = 1,2,\ldots, k$, and $(V_{i_1},\ldots, V_{i_k})\in X$.   Notice that for $(e,e') \in T$, such that $e\cap V_{i_j} = b$, $e'\cap V_{i_j} = b'$, $b \neq b'$, and $V_{i_j}$ lies completely inside a part in $\mathcal{Q}$, we have $|N(b)\triangle N(b')| \leq 2\delta$.  Therefore

$$|T| \leq K\left(\frac{n}{K}\right)^22\delta \leq  \frac{\varepsilon^2}{2Kk^2} n^2 {n\choose k-1} .$$

\noindent On the other hand, by Lemma \ref{pattern}, every $k$-tuple of parts $(V_{i_1},\ldots,V_{i_k})$ that is not $\varepsilon$-homogeneous gives rise to at least  $\varepsilon(1-\varepsilon)(n/K)^{k + 1}$ pairs $(e,e')$ in $T$.  Hence $|T| \geq |X| \varepsilon(1-\varepsilon)(n/K)^{k + 1}$.  Since $\varepsilon < 1/4$ and $k \geq 2$, the inequalities above imply that

$$|X| \leq  (2\varepsilon/3){K\choose k}.$$

\noindent Thus, the fraction of $k$-tuples of parts in $\mathcal{P}$ that are not $\varepsilon$-homogeneous is at most $\varepsilon/8 + 2\varepsilon/3 < \varepsilon$, and $K \leq c(1/\varepsilon)^{2d + 1}$ where $c = c(k,d)$.

Finally, it remains to show that the partition $\mathcal{P}$ can be computed in $O(n^k)$ time.  Given two vertices $s,v, \in V$, we have $|N(s) \triangle N(v)| = |N(s)| + |N(v)|  -  2|N(s) \cap N(v)|$.  Therefore we can determine if $|N(s) \triangle N(u)| < \delta$ in $O(n^{k-1})$ time.  Hence the maximal set $S\subset V$ described above (and therefore the partition $\mathcal{Q}$) can be computed in $O(n^{k})$ time since $|S| \leq n$.  The final equitable partition $\mathcal{P}$ requires an additional $O(n)$ time, which gives a total running time of $O(n^k)$. \end{proof}

We now establish Theorem \ref{tight} which shows that the partition size in Theorem \ref{reg1} is tight up to an absolute constant factor in the exponent.

\medskip

\begin{proof}[Proof of Theorem \ref{tight}]  Given two vertex subsets $X,Y$ of a graph $G $, we write $e_G(X,Y)$ for the number of edges between $X$ and $Y$ in $G$, and write $d_G(X,Y)$ for the density of edges between $X$ and $Y$, that is, $d_G(X,Y) = \frac{e_G(X,Y)}{|X||Y|}$.  The pair $(X,Y)$ is said to be $(\varepsilon, \delta)$-\emph{regular} if for all $X'\subset X$ and $Y'\subset Y$ with $|X'| \geq \delta |X|$ and $|Y'| \geq \delta |Y|$, we have $|d_G(X,Y) - d_G(X',Y')| \leq \varepsilon$.  In the case that $\varepsilon = \delta$, we just say \emph{$\varepsilon$-regular}.  We will make use of the following construction due to Conlon and Fox.

\begin{lemma}[\cite{CF12}]\label{nreg}
For $d\geq 16$ and $\varepsilon \in (0,1/100)$, there is a graph $H$ on $n = \lceil (5\varepsilon)^{-d/2}\rceil$ vertices such that for every equitable vertex partition of $H$ with at most $\sqrt{n}$ parts, there are at least an $\varepsilon$-fraction of the pairs of parts which are not $(4/5)$-regular.
\end{lemma}

Let $H = (V,E)$ be the graph obtained from Lemma \ref{nreg} on $n = \lceil (5\varepsilon)^{-d/2}\rceil$ vertices, where $\varepsilon \in (0,1/100)$ and $d \geq 16$, and consider a random subgraph $G\subset H$ by picking each edge in $E$ independently with probability $p = n^{-2/d} = 5\varepsilon$.  Then we have the following.

\begin{lemma}\label{probnreg}
In the random subgraph $G$, with probability at least $1 -n^{-2}$, every pair of disjoint subsets $X,Y \subset V$, with $|X| \leq |Y|$, satisfy

\begin{equation}\label{show}
|e_G(X,Y) - p\cdot e_H(X,Y)| < \sqrt{g},\end{equation}

\noindent where $g = 2|X||Y|^2\ln(ne/|Y|)$.

\end{lemma}

\begin{proof}
For fixed sets $X,Y \subset V(G)$, where $|X| = u_1$ and $|Y| = u_2$, let $E_H(X,Y) = \{e_1,\ldots, e_m\}$.  We define $S_i = 1$ if edge $e_i$ is picked and $S_i = 0$ otherwise, and set $S = S_1 +\cdots + S_m$.  A Chernoff-type estimate (see Theorem A.1.4 in \cite{AS}) implies that for $a > 0$, $Pr[|S-pm| > a] < 2e^{-2a^2/m}.$  Since $m \leq u_1u_2$, the probability that (\ref{show}) does not hold is less than $2e^{-2g/(u_1u_2)}$.  By the union bound, the probability that there are disjoint sets $X,Y \subset V(G)$ for which (\ref{show}) does not hold is at most

$$\begin{array}{ccl}
    \sum\limits_{u_2 = 1}^n \sum\limits_{u_1 = 1}^{u_2} {n\choose u_2}{n-u_2\choose u_1} 2e^{-2g/(u_1u_2)}& \leq &   \sum\limits_{u_2 = 1}^n \sum\limits_{u_1 = 1}^{u_2} \left(\frac{ne}{u_2}\right)^{u_2}\left(\frac{ne}{u_1}\right)^{u_1} 2e^{-2g/(u_1u_2)} \\\\
      & \leq &   \sum\limits_{u_2 = 1}^n \sum\limits_{u_1 = 1}^{u_2} 2\left(\frac{ne}{u_2}\right)^{-2u_2}  \leq n^{-2}.
  \end{array}$$

\end{proof}

By the analysis in Section \ref{lll}, the probability that $G$ has VC-dimension at least $d+1$ is at most $${n\choose d+1}n^{2^{d+1}}p^{(d+1)2^{d}} \leq n^{d + 1} n^{-2^{d+1}/d} <  \frac{1}{10},$$ since $d\geq 16$.  Therefore, the union bound implies that there is a subgraph $G \subset H$ such that $G$ has VC-dimension at most $d$, and every pair of disjoint subsets $X,Y \subset V$, with $|X|\leq |Y|$, satisfy

\begin{equation}\label{ineq}
|e_G(X,Y) - p\cdot e_H(X,Y)| < \sqrt{2|X||Y|^2\ln(ne/|Y|)}.
\end{equation}

We will now show that for every equitable vertex partition of $G$ into fewer than $\sqrt{n} = (5\varepsilon)^{-d/4}$ parts, there are at least an $\varepsilon$-fraction of the pairs of parts which are not $\varepsilon$-homogenous.

Let $\mathcal{P}$ be a equitable partition on $V$ into $t$ parts, where $t < \sqrt{n} = (5\varepsilon)^{-d/4}$.  By Lemma \ref{nreg}, there are at least $\varepsilon{t\choose 2}$ pairs of parts in $\mathcal{P}$ which are not $(4/5)$-regular in $H$.  Let $(X,Y)$ be such a pair.  Then there are subsets $X'\subset X$ and $Y'\subset Y$ such that $|X'| \geq 4|X|/5$, $|Y'|\geq 4|Y|/5$, and

$$|d_H(X,Y) - d_H(X',Y')| \geq 4/5.$$

\noindent Moreover, by (\ref{ineq}), we have

$$|e_G(X,Y) - p\cdot e_H(X,Y)| \leq \sqrt{2}\left(\frac{n}{t}\right)^{3/2}\ln(te) \leq  \frac{\sqrt{2}\ln(te)}{n^{1/4}}(n/t)^2.$$

\noindent Since $d \geq 16$ and $\varepsilon \in (0,1/100)$, this implies

$$|e_G(X,Y) - p\cdot e_H(X,Y)| \leq (5\varepsilon)^2\sqrt{2}\ln(te) (n/t)^2 \leq \frac{\varepsilon}{4}(n/t)^2. $$

\noindent Hence $|d_G(X,Y) - p\cdot d_H(X,Y)| \leq \varepsilon/4$.  Therefore we have

$$|d_G(X',Y') - d_G(X,Y)| \geq p\cdot |d_H(X',Y') - d_H(X,Y)| - 2\frac{\varepsilon}{4}  \geq 4\varepsilon -  \frac{\varepsilon}{2}  > 3\varepsilon.$$

Finally, it is easy to see that $(X,Y)$ is not $\varepsilon$-homogeneous in $G$.   Indeed if $(X,Y)$ were $\varepsilon$-homogeneous, then  we have either $d_G(X,Y) < \varepsilon$ or $d_G(X,Y) > 1-\varepsilon$.  In the former case we have $d_G(X',Y') > 3\varepsilon$, which implies $$e_G(X,Y) \geq e_G(X',Y') > 3\varepsilon \frac{4|X|}{5}\frac{4|Y|}{5} >\varepsilon|X||Y|,$$ contradiction.  In the latter case, we have $d(X',Y') < 1-3\varepsilon$, and a similar analysis shows that $e_G(X,Y) < (1 - \varepsilon)|X||Y|$, contradiction.

Thus, any equitable vertex partition on $G$ such that all but an $\varepsilon$-fraction of the pairs of parts are $\varepsilon$-homogeneous, requires at least $(5\varepsilon)^{-d/4}$ parts. \end{proof}

\section{Proof of Theorem \ref{jacob}}\label{graphramsey}

The family $\mathcal{G}$ of all \emph{complement reducible graphs}, or \emph{cographs}, is defined as follows:  The graph with one vertex is in $\mathcal{G}$, and if two graphs $G, H \in \mathcal{G}$, then so does their disjoint union, and the graph obtained by taking their disjoint union and adding all edges between $G$ and $H$.  Clearly, every induced subgraph of a cograph is a cograph, and it is well known that every cograph on $n$ vertices contains a clique or independent set of size $\sqrt{n}$.

Let $f_d(n)$ be the largest integer $f$ such that every graph $G$ with $n$ vertices and VC-dimension at most $d$ has an induced subgraph on $f$ vertices which is a cograph.  Cographs are perfect graphs, so that Theorem \ref{jacob} is an immediate consequence of the following result.

\begin{theorem}

For any $\delta \in (0,1/2)$ and for every integer $d\geq 1$, there is a $c = c(d,\delta)$ such that $f_d(n) \geq e^{c(\log n)^{1 - \delta}}$ for every $n$.

\end{theorem}

\begin{proof}
For simplicity, let $f(n) = f_d(n)$.  The proof is by induction on $n$.  The base case $n = 1$ is trivial.  For the inductive step, assume that the statement holds for all $n' < n$.  Let $\delta > 0$ and let $G = (V,E)$ be an $n$-vertex graph with VC-dimension $d$.  We will determine $c \in (0,1)$ later.

Set $\varepsilon =  (1/32)e^{-3c(\log n)^{1 - \delta}}$.  We apply Theorem \ref{reg1} to obtain an equitable partition $\mathcal{P}:V = V_1\cup \cdots \cup V_K$ into at most $K \leq \varepsilon^{-c_4}$ parts, where $c_4 = O(d)$, such that all but an $\varepsilon$-fraction of the pairs of parts are $\varepsilon$-homogeneous.  We call an unordered pair of distinct vertices $(u,v)$ \emph{bad} if at least one of the following holds:

 \begin{enumerate}

\item $(u,v)$ lie in the same part, or

\item $u \in V_i$ and $v \in V_j$, $i\neq j$, where $(V_i,V_j)$ is not $\varepsilon$-homogeneous, or

\item $u \in V_i$ and $v \in V_j$, $i\neq j$, $uv \in E(G)$ and $|E(V_i,V_j)| < \varepsilon |V_i||V_j|$, or

\item $u \in V_i$ and $v \in V_j$, $i\neq j$, $uv \not\in E(G)$ and $|E(V_i,V_j)| >(1- \varepsilon)|V_i||V_j|$.

\end{enumerate}

\noindent By Theorem \ref{reg1}, the number of bad pairs of vertices in $G$ is at most

$$K{n/K\choose 2} + \left(\frac{n}{K}\right)^2\varepsilon {K\choose 2}+ \varepsilon \left(\frac{n}{K}\right)^2 (1-\varepsilon){K\choose 2}\leq 2\varepsilon {n\choose 2 }.$$

By Tur\'an's Theorem, there is a subset $R\subset S$ of at least $\frac{1}{4\varepsilon}$ vertices such that $R$ does not contain any bad pairs.  This implies that all vertices of $R$ are in distinct parts of $\mathcal{P}$.  Furthermore, if $uv$ are adjacent in $R$, then the corresponding parts $V_i,V_j$ satisfy $|E(V_i,V_j)| \geq (1 - \varepsilon) |V_i||V_j|$, and if $uv$ are not adjacent, then we have $|E(V_i,V_j)| < \varepsilon |V_i||V_j|$.  Since the induced graph $G[R]$ has VC-dimension at most $d$, $G[R]$ contains a cograph $U_0$ of size $t = f(1/(4\varepsilon))$, which, by the induction hypothesis, is a set of size at least $e^{c(\log(1/4\varepsilon))^{1-\delta}}$.  Without loss of generality, we denote the corresponding parts of $U_0$ as $V_1,\ldots, V_t$.  Each part contains $n/K$ vertices.

For each vertex $u \in V_1$, let $d_b(u)$ denote the number of bad pairs $uv$, where $v \in V_i$ for $i = 2,\ldots, t$.  Then there is a subset $V'_1\subset V_1$ of size $\frac{n}{2K}$, such that each vertex $u \in V'_1$ satisfies $d_b(u)  < 8t\varepsilon (n/K)$.  Indeed, otherwise at least $n/(2K)$ vertices in $V_1$ satisfies $d_b(u)  \geq 8t\varepsilon (n/K)$, which implies

$$\frac{n}{2K}\frac{8t\varepsilon n}{K} \leq \sum\limits_{ u \in V'_1} d_b(u) \leq\sum\limits_{ u \in V_1} d_b(u) \leq \varepsilon (t-1)\left(\frac{n}{K}\right)^2,$$

\noindent and hence a contradiction.  By the induction hypothesis, we can find a subset $U_1 \subset V'_1$ such that the induced subgraph $G[U_1]$ is a cograph of size $f(n/(2K))$.  If the inequality

$$f\left(\frac{n}{2K}\right)  8t\varepsilon \frac{n}{K} > \frac{n}{4tK}$$

\noindent is satisfied, then we have

$$f^3(n) \geq f\left(\frac{n}{2K}\right) t^2 > \frac{1}{32\varepsilon}.$$

\noindent By setting $\varepsilon$ such that $\frac{1}{\varepsilon} = 32e^{3c(\log n)^{1 - \delta}}$, we have $f(n) \geq  e^{c(\log n)^{1 - \delta}}$ and we are done.

Therefore, we can assume that

$$f\left(\frac{n}{2K}\right)  8t\varepsilon \frac{n}{K} \leq \frac{n}{4tK}.$$

\noindent Hence, by deleting any vertex $v \in V_2\cup \cdots \cup V_t$ that is in a bad pair with a vertex in $U_1$, we have deleted at most $\frac{n}{4tK}$ vertices in each $V_i$ for $i = 2,\ldots, t$.

We repeat this entire process on the remaining vertices in $V_2,\ldots, V_t$.  At step $i$, we will find a subset $U_i\subset V_i$ that induces a cograph of size

$$f\left(\frac{n}{2K} - i\frac{n}{4Kt}\right) \geq   f\left(\frac{n}{4K}\right),$$

\noindent and again, if the inequality

$$f\left(\frac{n}{4K}\right)  8t\varepsilon \frac{n}{K} > \frac{n}{4tK}$$

\noindent is satisfied, then we are done by the same argument as above.  Therefore we can assume that our cograph $G[U_i]$ has the property that there are at most $n/(4tK)$ bad pairs between $U_i$ and $V_j$ for $j > i$.  At the end of this process, we obtain subsets $U_1,\ldots, U_t$ such that the union $U_1\cup \cdots \cup U_t$ induces a cograph of size at least $tf\left(\frac{n}{4K}\right)$. Therefore we have

\begin{equation}\label{series1}
\begin{array}{ccl}
    f(n)   & \geq & f\left( \frac{1}{4\varepsilon} \right)f\left(\frac{n}{4K}\right)\\\\
      & \geq & f\left(e^{3c(\log n)^{1-\delta} }  \right) f\left( e^{\log n - c\cdot c_5(\log n)^{1-\delta}}  \right)\\\\
     & \geq &  e^{c\left(3c (\log n)^{1-\delta}\right)^{1 - \delta}  }e^{c\left(\log n - c\cdot c_5(\log n)^{1 - \delta}\right)^{1-\delta}  },\\
  \end{array}
\end{equation}

\noindent where $c_5 = c_5(d)$. Notice we have the following estimate:

\begin{equation}\label{series2}
\begin{array}{ccl}
  \left(\log n - c\cdot c_5(\log n)^{1 - \delta}\right)^{1-\delta} & = & (\log n)^{1 - \delta}\left(1 - \frac{c\cdot c_5}{\log^{\delta}n}\right)^{1 - \delta} \\\\
   &  \geq & (\log n)^{1 - \delta}\left( 1 - \frac{c\cdot c_5}{(\log n)^{\delta}} \right)\\\\
   & \geq & (\log n)^{1 - \delta} - c\cdot c_5(\log n)^{1 - 2\delta}.  \\\\
\end{array}
\end{equation}

\noindent Plugging (\ref{series2}) into (\ref{series1}) gives

\begin{equation}
\begin{array}{ccl}
    f(n)   & \geq &  e^{c\left(3c(\log n)^{1-\delta} \right)^{1 - \delta} }\cdot e^{c(\log n)^{1 - \delta} - c^2\cdot c_5(\log n)^{1 - 2\delta}}\\\\
    & = & e^{c(\log n)^{1 - \delta}}\cdot e^{\left(3^{1-\delta}c^2(\log n)^{1 - 2\delta + \delta^2}  - c^2c_5(\log n)^{1-2\delta}  \right)}.
  \end{array}
\end{equation}

\noindent The last inequality follows from the fact that $c < 1$.  Let $n_0 = n_0(d,\delta)$ be the minimum integer such that for all $n \geq n_0$ we have

$$ 3^{1-\delta}(\log n)^{1 - 2\delta + \delta^2}  -  c_5(\log n)^{1-2\delta} \geq0. $$

\noindent We now set $c = c(d,t)$ to be sufficiently small such that the statement is trivial for all $n  < n_0$. Hence we have $f(n) \geq e^{c(\log n)^{1 - \delta}}$ for all $n$.\end{proof}

\section{Random constructions}\label{lll}

Here we prove Theorems \ref{offdiag} and \ref{notstrongeh}.  The proof of Theorem \ref{offdiag} uses the Lov\'asz Local Lemma \cite{ErLo} in a similar manner as Spencer \cite{Sp} to give a lower bound on Ramsey numbers.

\begin{lemma}[Lov\'asz Local Lemma]\label{local}
Let $\mathcal{A}$ be a finite set of events in a probability space.  For $A \in \mathcal{A}$ let $\Gamma(A)$ be a subset of $\mathcal{A}$ such that $A$ is independent of all events in $\mathcal{A} \setminus (\{A\}\cup\Gamma(A))$.  If there is a function $x:\mathcal{A} \rightarrow (0,1)$ such that for all $A \in \mathcal{A}$,

$$Pr[A] \leq x(A)\prod\limits_{B \in \Gamma(A)}(1 - x(B)),$$

\noindent then $Pr\left[\bigcap_{A \in \mathcal{A}} \overline{A}\right] \geq \prod\limits_{A \in \mathcal{A}}(1 - x(A)).$  In particular, with positive probability no event in $\mathcal{A}$ holds.

\end{lemma}

\begin{proof}[Proof of Theorem \ref{offdiag}]  Let $s$ and $d$ be positive integers such that $d> s + 2$.  Let $G(n, p)$ denote the random graph on $n$ vertices in which each edge appears with probability $p$ independently of all the other edges, where $p =  n^{-2/(s + 1)}$ and $n$ is a sufficiently large number.  For each set $S$ of $s$ vertices, let $A_S$ be the event that $S$ induces a complete graph.  For each set $T$ of $t$ vertices, let $B_T$ be the event that $T$ induces an empty graph.   Clearly, we have $Pr[A_S] = p^{s\choose 2}$ and $Pr[B_T] = (1 - p)^{t\choose 2}$.

For each set $D$ of $d$ vertices, let $C_D$ be the event that $D$ is shattered.  Then

$$\begin{array}{ccl}
  Pr[C_D] & \leq & \prod\limits_{W \subset D} Pr[\exists v \in V(G): N(v)\cap D = W]\\\\
   & = & \prod\limits_{W \subset D} \left(1 - \left(1 - p^{|W|}(1 - p)^{d - |W|}\right)^n\right) \\\\
   & = &  \prod\limits_{j = 0}^d \left(1 - \left(1 - p^{j}(1 - p)^{d - j}\right)^n\right)^{d\choose j}\\\\
   & \leq &  \prod\limits_{j = 1}^d \left(n\cdot p^{j}(1 - p)^{d - j}\right)^{d\choose j}\\\\
   & \leq &  \prod\limits_{j = 1}^d  n^{d\choose j}\cdot p^{j{d\choose j}} \\\\
      & \leq &    n^{2^d}\cdot p^{d2^{d-1}}. \\\\
\end{array}$$

Next we estimate the number of events dependent on each $A_S$, $B_T$ and $C_D$.  Let $S\subset V$ such that $|S| = s$.  Then the event $A_S$ is dependent on at most ${s\choose 2}{n\choose s-2} \leq s^2n^{s-2}$ events $A_{S'}$, where $|S'| = s$.  Likewise, $A_S$ is dependent on at most ${n\choose t}$ events $B_T$ where $|T| = t$.  Finally $A_S$ is dependent on at most ${s\choose 2}{n\choose d-2} \leq s^2n^{d-2}$ events $C_D$ where $|D| = d$.

Let $T \subset V$ be a set of vertices such that $|T| = t$.  Then the event $B_T$ is dependent on at most ${t\choose 2}{n\choose s-2} \leq t^2n^{s-2}$ events $A_S$ where $|S| = s$.  Likewise, $B_T$ is dependent on at most ${n\choose t}$ events $B_{T'}$ where $|T'| = t$.  Finally $B_T$ is dependent on at most ${t\choose 2}{n\choose d-2} \leq t^2n^{d-2}$ events $C_D$ where $|D| = d$.

Let $D \subset V$ be a set of vertices such that $|D| = d$.  Then the event $C_D$ is dependent on at most ${d\choose 2}{n\choose s-2}\leq d^2n^{s-2}$ events $A_S$ where $|S| = s$.  Likewise, $C_D$ is dependent on at most ${n\choose t}$ events $B_T$ where $|T| = t$.  Finally $C_D$ is dependent on at most ${d\choose 2}{n\choose d-2} \leq d^2n^{d-2}$ events $C_{D'}$ where $|D'| = d$.

By Lemma \ref{local}, it suffices to find three real numbers $x,y,z \in (0,1)$ such that

\begin{equation}\label{one}p^{s\choose 2} \leq x(1-x)^{s^2n^{s-2}}(1-y)^{n\choose t}(1 - z)^{s^2n^{d-2}},\end{equation}

\begin{equation}\label{two}(1-p)^{t\choose 2} \leq y(1-x)^{t^2n^{s-2}}(1-y)^{n\choose t}(1-z)^{t^2n^{d-2}},\end{equation}

\noindent and

\begin{equation}\label{three}n^{2^d}\cdot p^{d2^{d-1}} \leq z(1-x)^{d^2n^{s-2}}(1-y)^{n\choose t}(1 - z)^{d^2n^{d-2}}.\end{equation}

Recall $p =  n^{\frac{-2}{s + 1}}$, $s\geq 3$, and $d > s + 2$.  We now set $t = c_1n^{\frac{2}{s + 1}}(\log n)$,  $x = c_2 n^{\frac{-2{s\choose 2}}{s + 1}}$, $y = e^{-c_3n^{\frac{2}{s + 1}}(\log n)^2}$, and $z = c_4n^{2^d - \frac{2}{s+1}d2^{d-1}} $, where $c_1,c_2,c_3,c_4$ only depend on $s$ and $d$.    By letting $c_1 > 10c_3$, setting $c_1,c_2,c_3,c_4$ sufficiently large, an easy (but tedious) calculation shows that (\ref{one}), (\ref{two}), (\ref{three}) are satisfied when $n$ is sufficiently large.  By Lemma \ref{local}, there is an $n$-vertex $K_s$-free graph $G$ with VC-dimension at most $d$ and independence number at most $c_1 n^{\frac{2}{s + 1}}\log n$.\end{proof}

\medskip

\begin{proof}[Proof of Theorem \ref{notstrongeh}]  Let $d \geq 5$ and $n$ be a sufficiently large integer that will be determined later.  Consider the random $n$-vertex graph $G = G(n,p)$, where each edge is chosen independently with probability $p = n^{-4/d}$.  By choosing $n$ sufficiently large, the union bound and the analysis above implies that the probability that $G$ has VC-dimension at least $d$ is at most 1/3.

Let $A,B \subset V(G)$ be vertex subsets, each of size $k$.  The probability that $(A,B)$ is homogenous is at most

$$p^{k^2} + (1 - p)^{k^2}  \leq n^{-4k^2/d} + e^{-n^{-4/d}k^2}.$$

\noindent The probability that $G$ contains a homogeneous pair $(A,B)$, where $|A|,|B| = k$, is at most

$${n\choose k}{n - k \choose k}\left(  n^{-4k^2/d} + e^{-n^{-4/d}k^2}\right) < 1/3,$$

\noindent for $k =4n^{4/d}\log n$ and $n$ sufficiently large.  Thus, again by the union bound, there is a graph with VC-dimension less than $d$, with no two disjoint subsets $A,B \subset V(G)$ such that $(A,B)$ is homogeneous and $|A|,|B| = 4n^{4/d}\log n$.\end{proof}

\section{Ramsey-Tur\'an numbers for graphs with bounded VC-dimension}\label{ramseyturan}

In this section we prove Theorem \ref{rt}.  First let us recall a classical theorem in graph theory.

\begin{theorem}[Tur\'an]\label{turan3}

Let $G = (V,E)$ be a $K_p$-free graph with $n$ vertices.  Then the number of edges in $G$ is at most $\frac{1}{2}\left(1 - \frac{1}{p-1} + o(1)\right)n^2$.

\end{theorem}

Together with a sampling argument of Varnavides \cite{v}, we have the following lemma (see also Lemma 2.1 in \cite{keevash}).

\begin{lemma}\label{supersat}

For $\varepsilon > 0$, every $n$-vertex graph $G = (V,E)$ with $|E|\geq \frac{1}{2}\left(1 - \frac{1}{p-1}  + \varepsilon\right)n^2$ has at least $\delta n^p$ copies of $K_p$, where $\delta = \delta(p,\varepsilon)$.

\end{lemma}

In order to establish the upper bound in Theorem \ref{rt}, it suffices to show

$$\mathbf{RT}_d(n,K_{2p},o(n)) \leq \frac{1}{2}\left( 1 - \frac{1}{p-1}\right)n^2 + o(n^2),$$

\noindent since we have $\mathbf{RT}_d(n,K_{2p-1},o(n)) \leq \mathbf{RT}(n,K_{2p-1},o(n))$.  The following theorem implies the inequality above.

\begin{theorem}

Let $\varepsilon > 0$ and let $G = (V,E)$ be an $n$-vertex graph with VC-dimension $d$.  If $G$ is $K_{2p}$-free and $|E| >  \frac{1}{2}\left(1 - \frac{1}{p-1}  + \varepsilon\right)n^2$, then $G$ contains an independent set of size $\gamma n$, where $\gamma = \gamma(d,p,\varepsilon)$.

\end{theorem}

\begin{proof}

By Lemma \ref{supersat}, $G$ contains at least $\delta n^p$ copies of $K_p$, where $\delta = \delta(\varepsilon,p)$.  Without loss of generality, we can assume that $\delta$ is sufficiently small and will be determined later.  We apply the regularity lemma (Lemma \ref{reg1}) with approximation parameter $\delta/4$ to obtain a (near) equipartition $\mathcal{P}:V = V_1\cup\cdots\cup V_K$ such that $4/\delta \leq K\leq c\left(4/\delta\right)^{2d  +1}$, where $c = c(d)$, and all but a $\frac{\delta}{4}$-fraction of the pairs of parts in $\mathcal{P}$ are $(\delta/4)$-homogeneous.

By deleting all edges inside each part, we have deleted at most

$$K{n/K\choose 2} \leq \frac{n^2}{2K} \leq \frac{n^2}{8}\delta$$

\noindent edges.  By deleting all edges between pairs of parts that are not $(\delta/4)$-homogeneous, we have deleted an additional

$$\left(\frac{n}{K}\right)^2\frac{\delta}{4}{K\choose 2} \leq \frac{n^2}{8}\delta$$

\noindent edges.  Finally, by deleting all edges between pairs $(V_i,V_j)$ with density less than $\delta/4$, we have deleted at most

$$\frac{\delta}{4}\left(\frac{n}{K}\right)^2{K\choose 2} \leq \frac{n^2}{8}\delta$$

\noindent edges, which implies we have deleted in total less than $n^2\delta/2$ edges in $G$.  The only edges remaining in $G$ are edges between pairs of parts $(V_i,V_j)$ with density greater than $1-\frac{\delta}{4}$.  Since each edge lies in at most $n^{p-2}$ copies of $K_p$, we have deleted at most $\delta n^p/2$ $K_p$-s in $G$.  Therefore there is at least one copy of $K_p$ remaining, which implies that there are $p$ parts $V_{i_1},\ldots,V_{i_p} \in \mathcal{P}$ that pairwise have density at least $1-\frac{\delta}{4}$, with $|V_{i_j}| = n/K$.  Set $\delta_1 = \delta/4$.

For fixed $j \in \{2,\ldots, p\}$, notice that there are at least $(1 - 1/(2p))(n/K)$ vertices $v \in V_{i_1}$ such that $|N(v)\cap V_{i_j}| \geq (1- 4\delta_1  p)n/K$.  Indeed, otherwise we would have

$$|E(V_{i_1},V_{i_j})| \leq \left(n/K\right)\left(1 - \frac{1}{2p}\right)\left(\frac{n}{K}\right)  +  \frac{n/K}{2p} \left(1 - 4\delta_1  p\right)\frac{n}{K} = \left(\frac{n}{K}\right)^2 -  2\delta_1 (n/K)^2 .$$

\noindent On the other hand, $|E(V_{i_1},V_{i_j})|\geq \left(1- \delta_1 \right)(n/K)^2$.  This implies $2\delta_1 < \delta_1$ which is a contradiction.

Therefore there is a subset $V'_{i_1}\subset V_{i_1}$ with $|V'_{i_1}| \geq |V_{i_1}|/2$ such that each vertex $v \in V'_{i_1}$ satisfies $|N(v)\cap V_{i_j}| \geq (1-4\delta_1 p)|V_{i_j}|$ for all $j = 2,\ldots, p$.  If $V'_{i_1}$ is an independent set, then we are done since $|V'_{i_1}| \geq n/(2K)$.  Otherwise we have an edge $uv$ in $V'_{i_1}$.  For $j = 2,\ldots, p$, the pigeonhole principle implies that $|V_{i_j}\cap N(u)\cap N(v)|\geq \frac{n}{K}(1 - 8\delta_1p)$.  We define $V^{(2)}_{i_j}$ to be a set of exactly $\frac{n}{K}(1 - 8\delta_1p)$ elements in $V_{i_j}\cap N(u)\cap N(v)$.  Notice that the graph induced on the vertex set $V^{(2)}_{i_2} \cup \cdots \cup V^{(2)}_{i_p}$ is $K_{2p-2}$-free.  Moreover, the density between each pair of parts $(V^{(2)}_{i_j},V^{(2)}_{i_{\ell}})$ is at least $(1 - \delta_2)$ where $\delta_2 = \delta_1 + 16\delta_1p$.  We repeat this process on the remaining $p-1$ parts $V^{(2)}_{i_2}, \ldots, V^{(2)}_{i_p}$.

After $j$ steps, we have either found an independent set of size at least $$\frac{n}{2K}(1 - 8\delta_1p)(1 - 8\delta_2(p-1)) \cdots (1 - 8\delta_{j-1}(p-j+2)),$$ where $\delta_k$ is defined recursively as $\delta_1 = \delta/4$ and $\delta_k = \delta_{k-1} + 16\delta_{k-1}p$, or we have obtained subsets $V^{(j)}_{i_j} ,\ldots,V^{(j)}_{i_p}$ such that $$|V^{(j)}_{i_{\ell}}| = \frac{n}{K}(1 - 8\delta_1p)(1 - 8\delta_2(p-1)) \cdots (1 - 8\delta_{j-1}(p-j)),$$ for $\ell = j,\ldots, p$, $V^{(j)}_{i_j} \cup \cdots \cup V^{(j)}_{i_p}$ is $K_{2p- 2j}$-free, and the density between each pair of parts $(V^{(j)}_{i_{k}}, V^{(j)}_{i_{\ell}})$ is at least $1 - \delta_j$.

By letting $\delta  = \delta(\varepsilon,p)$ be sufficiently small such that $\delta_k < \frac{1}{100p}$ for all $k\leq p$, we obtain an independent set of size $\gamma n$, where $\gamma = \gamma(d,p,\varepsilon)$.

\end{proof}

The lower bound on $\mathbf{RT}_d(n,K_{2p-1},o(n))$ and $\mathbf{RT}_d(n,K_{2p},o(n))$ in Theorem \ref{rt} follows from a geometric construction of Fox et al. in \cite{FPS15} (see page 15), which is a graph with VC-dimension at most four.

\section{Concluding remarks}

Many interesting results arose in our study of graphs and hypergraphs with bounded VC-dimension.  In particular, we strengthen several classical results from extremal hypergraph theory for hypergraphs with bounded VC-dimension.  Below, we briefly mention two of them.

\medskip

\noindent \textbf{Hypergraphs with bounded VC-dimension.}  Erd\H os, Hajnal, and Rado \cite{EHR} showed that every $3$-uniform hypergraph on $n$ vertices contains a clique or independent set of size $c\log\log n$.  A famous open question of Erd\H os asks if $\log\log n$ is the correct order of magnitude for Ramsey's theorem for 3-uniform hypergraphs.  According to the best known constructions, there are 3-uniform hypergraphs on $n$ vertices with no clique or independent set of size $c'\sqrt{\log n}$.   For $k\geq 4$, the best known lower and upper bounds on the size of the largest clique or independent set in every $n$-vertex $k$-uniform hypergraph is of the form $c\log^{(k-1)} n$ (the $(k-1)$-times iterated logarithm) and $c' \sqrt{\log^{(k - 2)} n} $, respectively (see \cite{CFS10} for more details). By combining Theorem \ref{jacob} with an argument of Erd\H os and Rado \cite{ER}, one can significantly improve these bounds for hypergraphs of bounded (neighborhood) VC-dimension.

  \begin{theorem}\label{hypramsey}
Let $k\geq 3$ and $d \geq 1$. Every $k$-uniform hypergraph on $n$ vertices with VC-dimension $d$ contains a clique or independent set of size $e^{\left(\log^{(k-1)} n\right)^{1 - o(1)}}$.
  \end{theorem}

Geometric constructions given by Conlon et al.~\cite{CFPSS} show that Theorem \ref{hypramsey} is tight apart from the $o(1)$ term in the second exponent.  That is, for fixed $k\geq 3$, there are $k$-uniform hypergraphs on $n$ vertices with VC-dimension $d = d(k)$ such that the largest clique or independent set is of size $O(\log^{(k-2)} n)$.

\medskip

\noindent \textbf{The Erd\H os-Hajnal conjecture for tournaments.}  A \emph{tournament} $T  = (V,E)$ on a set $V$ is an orientation of the edges of the complete graph on the vertex set $V$, that is, for $u,v \in V$ we have either $(u,v) \in E$ or $(v,u) \in E$, but not both.  A tournament with no directed cycle is called \emph{transitive}.  If a tournament has no subtournament isomorphic to $T$, then it is called $T$-free.

An old result due to Entringer, Erd\H os, and Harner \cite{EEH} and Spencer \cite{S74} states that every tournament on $n$ vertices contains a transitive subtournament of size $c\log n$, which is tight apart from the value of the constant factor.  Alon, Pach, and Solymosi \cite{APS} showed that the Erd\H os-Hajnal conjecture is equivalent to the following conjecture.

\begin{conjecture}\label{conjeht}
For every tournament $T$, there is a positive $\delta = \delta(T)$ such that every $T$-free tournament on $n$ vertices has a transitive subtournament of size $n^{\delta}$.

\end{conjecture}

In particular, it is known that every $T$-free tournament on $n$ vertices contains a transitive subtournament of size $e^{c\sqrt{\log n}}$, where $c = c(T)$.  Here we note that this bound can be improved in the special case that the forbidden tournament $T = (V,E)$ is \emph{2-colorable}, that is, there is a 2-coloring on $V(T)$ such that the each color class induces a transitive subtournament.

\begin{theorem}\label{eht}

For fixed integer $k>0$, let $T$ be a 2-colorable tournament on $k$ vertices.  Then every $T$-free tournament on $n$ vertices contains a transitive subtournament of size $e^{(\log n)^{1 - o(1)}}$.

\end{theorem}

\noindent The idea of the proof of Theorem \ref{eht} is to use the fact that a tournament $T$ is 2-colorable if and only if the outneighborhood set system of every $T$-free tournament has VC-dimension at most $c(T)$.  There is a straightforward analogue of Theorem \ref{reg1} for tournaments whose outneighborhood set system has bounded VC-dimension, and with this analogous tool, the proof of Theorem \ref{eht} is essentially the same as the proof of Theorem \ref{jacob}.

\section{Acknowledgements}

We would like to thank Lisa Sauermann for pointing out a small error in an earlier version of the proof of Lemma \ref{pattern}.

\end{document}